\documentclass{birkjour}

\usepackage[utf8]{inputenc}
\usepackage{amsmath,amssymb,amsfonts,color, enumerate, amsthm}
\usepackage{mathrsfs}
\usepackage{graphicx}
\usepackage{lineno}
\usepackage{array}
\usepackage{longtable}
\usepackage{natbib}
\usepackage{url}
\usepackage{soul}
\usepackage{hyperref}
\usepackage{xcolor}
\usepackage{lipsum}
\usepackage{blkarray}
\usepackage{orcidlink}

\newcommand{\ip}[2]{\left \langle #1,\ #2\right \rangle}

\DeclareMathOperator*{\esssup}{ess\,sup}
\DeclareMathOperator*{\essinf}{ess\,inf}

\newcommand{\R}{\mathbb{R}}      
\newcommand{\C}{\mathbb{C}}      
\newcommand{\D}{\mathbb{D}}	     
\newcommand{\T}{\mathbb{T}}	     

\newtheorem{theorem}{Theorem}[section]

\newtheorem{proposition}[theorem]{Proposition}

\newtheorem{lemma}[theorem]{Lemma}
\newtheorem{remark}[theorem]{Remark}

\newtheorem{question}[theorem]{Question}

\begin{document}

\title[Spectra for Toeplitz Operators associated with $\mathscr{A}_{a,b}$]{Spectra  for Toeplitz Operators associated with a  Constrained Subalgebra
}

\author[Felder]{Christopher Felder\,\orcidlink{0000-0002-0736-0987}}
\address{Department of Mathematics and Statistics, Washington University In St. Louis, St. Louis, Missouri, 63130}
\email{cfelder@wustl.edu}
\thanks{The first named author was supported in part by NSF Grant DMS-1565243.}

\author[Pfeffer]{Douglas T. Pfeffer\,\orcidlink{0000-0001-8287-2686}}
\thanks{}
\address{Department of Mathematics and Computer Science, Berry College, PO Box 495014, Mount Berry GA 30149-5014}
\email{dpfeffer@berry.edu}

\author[Russo]{Benjamin P. Russo\,\orcidlink{0000-0002-6089-0696}}
\address{Computer Science and Mathematics Division, Oak Ridge National Laboratory, Oak Ridge, TN 37831, United States}
\email{russobp@ornl.gov}
\thanks{Notice: This manuscript has been authored, in part, by UT-Battelle, LLC, under contract DE-AC05-00OR22725 with the US Department of Energy (DOE). The US government retains and the publisher, by accepting the article for publication, acknowledges that the US government retains a nonexclusive, paid-up, irrevocable, worldwide license to publish or reproduce the published form of this manuscript, or allow others to do so, for US government purposes. DOE will provide public access to these results of federally sponsored research in accordance with the DOE Public Access Plan (\href{http://energy.gov/downloads/doe-public-access-plan}{http://energy.gov/downloads/doe-public-access-plan}).}

\subjclass{Primary 47B35; Secondary 47A10, 47A75}

\keywords{Toeplitz operator, constrained subalgebra, spectrum, point spectrum}

\date{}
\maketitle

\begin{abstract}
A two-point algebra is a set of bounded analytic functions on the unit disk that agree at two distinct points $a,b \in \D$. This algebra serves as a multiplier algebra for the family of Hardy Hilbert spaces $H^2_t := \{ f\in H^2 : f(a)=tf(b)\}$, where $t\in \C\cup\{\infty\}$. We show that various spectra of certain Toeplitz operators acting on these spaces are connected. 
\end{abstract}

\section{Introduction and Motivation}\label{intro}
For $1 \le p \le \infty$, let $L^p = L^p(\T, \mu)$ denote the classical Lebesgue spaces on the unit circle $\T$ with respect to normalized Lebesgue measure $\mu$. Similarly, let $H^p$ denote the usual Hardy spaces. 
Recall the standard identification of $H^p$ as spaces of analytic functions on the unit disk $\D$ such that
\[
\sup_{0 < r < 1} \ \int_0^{2\pi} \left|f\left(re^{it}\right)\right|^p \, dt < \infty.
\]
When $p=\infty$, these spaces are comprised of (essentially) bounded functions. 

For $\phi \in L^\infty$, the classical \textit{Toeplitz operator with symbol $\phi$} acts on $H^2$ by $T_\phi f = P_+(\phi f)$, where $P_+$ is the orthogonal projection from $L^2$ onto $H^2$. 

In 1963, Halmos asked if the spectrum of every Toeplitz operator, $\sigma(T_\phi)$, is connected \cite{halmos}. 
At that time, the following facts were known about the spectrum $\sigma$ of Toeplitz operators, due to Hartman and Wintner (e.g. see \cite[Chapter 7]{banachtech}):
\begin{itemize}
    \item[-] If $\phi \in L^\infty$ is real-valued, then $\sigma(T_\phi) = [\essinf(\phi), \esssup(\phi)]$.
    \item[-] If $\phi \in H^\infty$, then $\sigma(T_\phi)=\overline{\phi(\D)}$.
\end{itemize}  
Shortly after, Widom gave an affirmative answer to Halmos' question with the following result \cite{widomspectrum}:
\begin{itemize}
    \item[-] Every Toeplitz operator has connected spectrum and connected essential spectrum.
\end{itemize}
It can also be shown that when the symbol of the Toeplitz operator is real-valued, its point spectrum $\sigma_p$ is empty, and therefore connected (e.g. see \cite[Exercise 12.4.3]{theoryofhb}).
We point the reader to \cite{MR2743417} for a well-written account of this history. 

The aim of this paper is to present results in this vein, but for  Toeplitz operators acting on sub-Hardy Hilbert spaces arising naturally from certain finite codimension subalgebras of $H^\infty$; so-called \textit{constrained subalgebras}.

The most well-known constrained subalgebra of $H^\infty$ is the Neil algebra,
\[
\mathfrak{A} := \{  f\in H^\infty :  f'(0)=0 \}.
\]
Note that $H^\infty$ acts as a multiplier algebra for $H^2$; for each $\phi \in H^\infty$, we have $\phi H^2 \subseteq H^2$. When moving to the Neil  setting, we have that $\mathfrak{A}$ serves as a multiplier algebra not for a single space, but for a continuum of spaces
\[
H^2_{\alpha,\beta} = \{  f\in H^2 :  \alpha f(0) = \beta f'(0)  \textrm{ for }  (\alpha,\beta)\in \mathbb{S}^2 \},
\]
where $\mathbb{S}^2$ is the complex unit-sphere in $\C^2$ (see \cite{DPRS} for details). Although $\mathfrak{A}$ serves as a multiplier algebra for many Hilbert spaces of analytic functions, 
it is canonical to consider $\mathfrak{A}$ acting on
the spaces $H^2_{\alpha,\beta}$, as the associated representations of $\mathfrak{A}$ are rank one bundle shifts in this setting (see \cite[Section 4]{broschinski2014eigenvalues}). This perspective is useful in many contexts, however, is not needed for this paper.

For $\phi \in L^\infty$, the Toeplitz operator $T^{\alpha,\beta}_\phi$ acts on $H^2_{\alpha,\beta}$  by
\[
T^{\alpha,\beta}_\phi f = P_{\alpha,\beta} (\phi f),
\]
where $P_{\alpha,\beta}$ is the orthogonal projection from $L^2$ onto $H^2_{\alpha,\beta}$. 

In the Neil setting, Broschinski  \cite{broschinski2014eigenvalues} observed that the point spectrum of $T^{\alpha,\beta}_\phi$ might be non-empty; a stark difference from the classical setting. However, in a Widomesque quest, Broschinksi showed that the point spectrum of these operators \textit{relative to the algebra $\mathfrak{A}$} is indeed connected for certain symbols. In particular, Broschinski proved that when $\phi\in L^\infty$ is real-valued, the set of \textit{eigenvalues of $T^{\alpha,\beta}_\phi$ relative to $\mathfrak{A}$},
\[
\Lambda_\phi^{\mathfrak{A}} := \bigcup_{(\alpha,\beta)\in \mathbb{S}^2} \sigma_p(T^{\alpha,\beta}_\phi)
\]
is either empty, a point, or an interval.
Heuristically speaking, when working in settings that involve infinite families of representation-carrying Hardy spaces, it is typical for results to involve the entire family (e.g. see discussion in Section \ref{background}).

This paper provides, among other results, an analogue of Broschinski's result in the setting of a so-called two-point algebra.
The two-point algebra associated to fixed points $a,b\in \D$ is
\[
\mathscr{A}_{a,b} := \{  f\in H^\infty(\D) : f(a)=f(b)\}.
\]
Similar to the Neil algebra, there is an associated infinite family of sub-Hilbert Hardy spaces that each carry a representation for $\mathscr{A}_{a,b}$. For fixed $a,b\in \D$, define
\[
H^2_t := \{ f\in H^2 : f(a) = tf(b) \},
\]
where $t\in \widehat{\C} := \C \cup \{ \infty \}$. Note that, while the parameter space $\C\setminus\{0\}$ could also be used to yield the main results in this paper, we instead use the  Riemann sphere $\widehat{\C}$ to be consistent with the notation introduced in \cite{pfefferjury}, where compactness of the parameter space was necessary. To this end, this paper considers $t\neq 0,\infty$, and leaves the details of these cases to the interested reader. For each choice of $t\in \widehat{\C}$, the space $H^2_t$ carries a representation for $\mathscr{A}_{a,b}$.
Specifically, the map taking $h \in \mathscr{A}_{a,b}$ to the operator on $H^2_t$ given by $f \mapsto h f$, is an isometric homomorphism from $\mathscr{A}_{a,b}$ to the collection of bounded linear operators on $H^2_t$. In particular, we have $\mathscr{A}_{a,b}H^2_t\subseteq H^2_t$ for every $t\in \widehat{\C}$. As with the Neil algebra, the representations associated with $H^2_t$ can be seen as rank one bundle shifts for $\mathscr{A}_{a,b}$.

For $\phi \in L^\infty$, we define the Toeplitz operator $T^t_\phi\colon H^2_t\to H^2_t$ by 
\[
T^t_\phi f = P_{t}(\phi f),
\]
where $P_{t}$ is the orthogonal projection from $L^2$ onto $H^2_t$. We denote the \textit{eigenvalues of $T^t_\phi$ relative to $\mathscr{A}_{a,b}$} as
\[
\Lambda_\phi^{a,b} := \bigcup_{t\in \widehat{\C}} \sigma_p(T^{t}_\phi).
\]
In order to avoid trivialities when discussing $\Lambda_\phi^{a,b}$, it is always assumed that the symbol $\phi$ is non-constant. 

In this paper, we establish the following main results, regarding analytic and real-valued symbols, respectively:

\vspace{.5pc}

\noindent \textbf{Theorem \ref{mainresult1}.} \textit{If $\phi\in \mathscr{A}_{a,b}$, then
\begin{enumerate}[(i)]
    \item $\sigma(T^t_\phi) = \overline{\phi(\D)}$ and \label{main1item1}
    \item $\Lambda_\phi^{a,b} = \{ \phi(a)\}$. \label{main1item2}
\end{enumerate} In particular, both $\sigma(T^t_\phi)$ and $\Lambda_\phi^{a,b}$ are connected.}

\vspace{.5pc}

\noindent \textbf{Theorem \ref{mainresult2}.} \textit{If $\phi\in L^\infty$ is real-valued, then $\Lambda_\phi^{a,b}$ is either empty, a point, or an interval. In particular, $\Lambda_\phi^{a,b}$ is connected.}

\vspace{.5pc}

\section{Background}\label{background}
Before moving to results, we provide further background on the spectral theory of Toeplitz operators, as well as some relevant information on constrained subalgebras. 

Investigations into the spectra of Toeplitz operators are diverse and far-reaching, beginning with spectral inclusion theorems for self-adjoint Toeplitz operators, which originated in \cite{hw1, hw2, brownhalmos}. 
Results involving the spectra of Toeplitz operators with continuous or piece-wise continuous symbols can be found in \cite{contsymbol1, contsymbol2, contsymbol3, pwcontsymbol1, dev1}.
The spectrum of a Toeplitz operator with analytic symbol was first investigated in \cite {analyticsymbol1}. Halmos initially posed the connected-spectrum question in \cite{halmos}, which was answered by Widom in \cite{widomspectrum, widomspectrum2}.

Following these investigations, there were several works aimed at exploring the spectra of Toeplitz operators acting on spaces other than $H^2$. For example, Devinatz \cite{dev1} examined the spectra of real-symboled Toeplitz operators defined on a Hilbert space associated with a Dirichlet algebra. Another natural theme in this arena has been to consider Toeplitz operators acting on Hilbert spaces of analytic functions defined on assorted domains. Early work on multiply-connected planar domains was due to Abrahamse, where he established the existence of self-adjoint Toeplitz operators (relative to a single model space) with disconnected point spectrum \cite[Section 5]{AbrahamseToeplitz}. Further investigations include \cite{clancey1, aryana1, akbari1, clancey2}, where more in-depth spectral analyses were conducted, including the study of large numbers of isolated eigenvalues in the gaps of the range of the symbol of the operator. The case of the annulus was further discussed in \cite{broschinski2014eigenvalues}, where a version of Theorem \ref{mainresult2} is established. Additional work on multiply connected domains can be found in  \cite{Ballconstrained, Mult-ConnDom}.

The introduction of Toeplitz operators to the setting of constrained subalgebras came about, in part, through the study of Pick interpolation. In short, the Pick interpolation problem asks for a multiplier, of norm less than or equal to one, that maps a set of initial points to a set of target points. It is natural to ask the interpolating function to obey additional algebraic constraints. In this direction, Pick interpolation on the Neil algebra was investigated in \cite{DPRS}, and on a generalized two-point algebra in \cite{Raghupathi2point}. As a result, further work on the associated Toeplitz operators has been carried out. In \cite{Anderson-ROchberg}, Anderson and Rochberg established a Widom-type invertibility theorem for Toeplitz operators associated with constrained subalgebras of the unit disk. In \cite{pfefferjury}, Jury and the second named author established  Widom and Szeg\H{o} theorems for finite codimensional subalgebras of a class of uniform algebras defined on finite (connected) Riemann surfaces-- of which the Neil and two-point algebras are the prototypical examples. Such Szeg\H{o} and Widom theorems for Toeplitz operators on the Neil algebra were first discussed in \cite{balasubramanian2019szegHo}. For more results related to constrained algebras, see \cite{Ballconstrained, dritschelconstrained, Raghupathi2point, Trent1, Trent2, Banjade}.

\section{Structure of $H^2_t$} \label{preliminaries}
We now move to establish the results of the paper, starting with some structural observations of $H^2_t$ spaces. First, recall that $H^2$ is a reproducing kernel Hilbert space on the unit disk $\D$. That is, for each $f\in H^2$ and $w \in \D$, we have 
\[
f(w) = \langle f, k_w \rangle := \int_\T f \overline{k_w} \, d\mu,
\]
where, for $w,z \in \D$,
\[
k_w(z) = \frac{1}{1-\overline{w}z}.
\]
The function $k_w$ is known as the \textit{Szeg\H{o} kernel at $w$}.
The spaces $H^2_t$ inherit the reproducing property from $H^2$, and we denote their reproducing kernels by $k_w^t$. We will also need the Blaschke product at $a$ and $b$, given by
\[
B_{a,b}(z) = \frac{z - a}{1 - \overline{a}z}\frac{z - b}{1 - \overline{b}z}.
\]

We take the convention that $0\cdot \infty = 0$ on $\widehat{\C}$.

\begin{proposition} \label{onbasis}
Let $t \in \C\setminus\{0\}$. For each $H^2_t$ space associated to $a,b \in \D$, we have
\begin{enumerate}[(i)]
\item $k_a^t = \overline{t}k_b^t$. \label{onbasis1}
\item $H^2_t =\C k_a^t \oplus B_{a,b}H^2$. \label{onbasis2}
\item The set $\{k_a^t/\|k_a^t\|\}\cup \{B_{a,b} z^n\}_{n\ge 0}$ is an orthonormal basis for $H^2_t$. \label{onbasis3}
\end{enumerate}

\end{proposition}

\begin{proof}
Point (\ref{onbasis1}) follows directly from definition. 
In order to establish (\ref{onbasis2}), we show $B_{a,b}H^2$ is the orthogonal complement of  $\C k_a^t$ in $H^2_t$. Suppose $f\in H^2_t$ with $\langle f, k_a^t \rangle =0$. Then $\langle f, \overline{t}k_b^t \rangle =0$, and we see $f$ vanishes at $a$ and $b$, so $f \in B_{a,b}H^2$. Conversely, suppose $f \in B_{a,b}H^2$. Since $B_{a,b}H^2 \subset H^2_t$, we have $f \in H^2_t$ and $f(a) = 0$. Thus, $f$ is orthogonal to $\C k_a^t$, and the result follows. Point (\ref{onbasis3}) now follows immediately by noting that, in $H^2$, multiplication by $B_{a,b}$ is an isometry and the monomials are orthogonal.

\end{proof}

Using the above proposition, a precise formula for the reproducing kernel can be given. Specifically,
\[
k^t_a(z) = \frac{\overline{t}}{\overline{t} - \tau}(k_a(z) - \tau k_b(z)),
\]
where $\tau = \frac{k_a(a) - tk_a(b)}{k_b(a) - tk_b(b)}$. Further, 
for $z,w \in \D$, the reproducing kernel for $H^2_t$ is given by
\[
k^t_w(z) = \frac{\overline{k^t_a(w)}}{\|k^t_a\|^2}k^t_a(z) + \overline{B_{a,b}(w)}B_{a,b}(z)k_w(z).
\]

We end this section with a proposition about the inner-outer factorization of a function $g\in \ker(T^t_\phi)$. Specifically, we note that the outer factor of such a function $g$ must also be in the kernel of a Toeplitz operator (with a possibly different parameter $t$).

\begin{proposition}\label{praisejesus}
Let $\phi \in L^\infty$ and suppose $g\in \operatorname{ker}(T_\phi^t)$. Put $g = \theta G$ with $\theta$ inner and $G$ outer. Then $G\in \operatorname{ker}(T_\phi^s)$, where $s = G(a)/G(b)$ (i.e. $G\in H^2_s$). 
\end{proposition}

\begin{proof}
Let $h\in H^2_s$ and note as $G$ is outer, it is non-vanishing on $\D$, which means $s\in \C \setminus \{0\}$. Since $g=\theta G \in H^2_t$ and $G\in H^2_s$, it follows that $\theta\in H^2_{t/s}$. In turn, $\theta h\in H^2_t$ and we have
\[
    \ip{T_\phi^s G}{h} =\int_\T \phi G\overline{h}\, d\mu 
    =\int_\T \phi g \overline{\theta h}\, d\mu 
    =\ip{T^t_\phi g}{\theta h}
    =0.
\]
As this holds for all $h \in H^2_s$, we have $G\in \ker(T^s_\phi)$. 
\end{proof}

\section{Toeplitz operators with analytic symbols} 

In this section, we provide a proof of Theorem \ref{mainresult1}. We begin by showing that the point spectrum of $T^t_\phi$ with analytic symbol is a singleton. Recall that $P_{t}$ is the orthogonal projection from $L^2$ onto $H^2_t$.
\begin{proposition} \label{analyticpointspectrum}
If $\phi \in H^\infty$ is non-constant, then
\[
\sigma_p(T_\phi^t) = \langle \phi k_a^t, k_a^t \rangle/\|k_a^t\|^2.
\]
\end{proposition}

\begin{proof}
Suppose that $(T_\phi^t - \lambda I)g = 0$ for $\lambda \in \C$ and $g \in H^2_t\setminus\{0\}$. Using Proposition \ref{onbasis}, we can put $g = c k_a^t + B_{a,b}h$, for some $c \in \C$ and $h \in H^2$ to obtain
\begin{align*}
    0 &= T_\phi^t(c k_a^t + B_{a,b}h) -  \lambda(c k_a^t + B_{a,b}h)\\
    &= P_{t}(\phi (ck_a^t + B_{a,b}h)) -  \lambda(c k_a^t + B_{a,b}h)\\
    &= cP_{t}(\phi k_a^t) + P_{t}(\phi B_{a,b}h) -  \lambda(c k_a^t + B_{a,b}h)\\
    &= cP_{t}(\phi k_a^t) + \phi B_{a,b}h -  \lambda(c k_a^t + B_{a,b}h),
\end{align*}
where the last equality follows from the fact that $\phi B_{a,b}h \in B_{a,b}H^2 \subseteq H^2_t$. 
Now put $P_{t}(\phi k_a^t) = c'k_a^t + B_{a,b}f$ where $c':= \langle \phi k_a^t, k_a^t \rangle/\|k_a^t\|^2$ and $f\in H^2$. 
Then from the above equality we have
\[
0 = c(c'k_a^t + B_{a,b}f) + \phi B_{a,b}h -  \lambda(c k_a^t + B_{a,b}h)
\]
which implies that 
\[
B_{a,b}((\lambda - \phi)h - cf)= c(c' - \lambda)k_a^t.
\]
As the left hand side is in $B_{a,b}H^2$ and right hand side is in $\textrm{span}\{k_a^t\}$, which are orthogonal, we see that  both sides must equal zero. 
Considering that $c(c' - \lambda)k_a^t = 0$, i.e. $c(c' - \lambda) = 0$, we have $c = 0$ or $(c' - \lambda) = 0$.

If $c = 0$, this means that $g = B_{a,b}h$, which we will show cannot happen. Referring back to the first set of equalities in this proof, we then get that 
\[
    0 = \phi B_{a,b}h - \lambda B_{a,b}h 
      = (\phi - \lambda)B_{a,b}h.
\]
This now says that either $\phi = \lambda$, which is a contradiction, or that $h= 0$. But if $h=0$, then $B_{a,b}h = g = 0$,  which is also a contradiction. 

In turn, it must be that $c' - \lambda = 0$, which implies that $\lambda = c' = \langle \phi k_a^t, k_a^t \rangle/\|k_a^t\|^2$, as claimed. 
\end{proof}

\begin{remark}
We point out that the quantity $\langle \phi k_a^t, k_a^t \rangle/\|k_a^t\|^2$ is known as the Berezin transform of $T_\phi^t$ at $a$. This transform appears naturally in the study of various operator theoretic properties of Toeplitz operators (e.g. see \cite[Section 2]{cowen2021convexity}). 
\end{remark}

We now provide the main result of this section, which completely characterizes the spectrum and relative eigenvalues  of $T^t_\phi$ when the symbol is in our multiplier algebra.

\begin{theorem}\label{mainresult1} If $\phi\in \mathscr{A}_{a,b}$, then
\begin{enumerate}[(i)]
    \item $\sigma(T^t_\phi) = \overline{\phi(\D)}$ and
    \item $\Lambda_\phi^{a,b} = \{ \phi(a)\}$.
\end{enumerate}
In particular, both $\sigma(T^t_\phi)$ and $\Lambda_\phi^{a,b}$ are connected.
\end{theorem}
\begin{proof}
To see (\ref{main1item1}), note as $\phi\in \mathscr{A}_{a,b}$, we have $T^t_\phi f = \phi f$. Now let $z_0 \in \D$ and suppose that $\lambda=\phi(z_0)$. Observe, for any $f\in H_t^2$, that
\[
\left((T_\phi^t-\lambda I)f\right)(z_0)=(\phi(z_0)-\lambda)f(z_0)=0.
\]
Thus, $T_\phi^t-\lambda$ cannot be surjective, since all functions in its range must have a zero at $z_0$. This shows $\phi(\mathbb{D})\subseteq \sigma(T_\phi^t)$. Since the spectrum is compact and hence closed, we have $\overline{\phi(\mathbb{D})}\subseteq \sigma(T_\phi^t)$. To see the inclusion $\sigma(T^t_\phi) \subseteq \overline{\phi(\mathbb{D})}$, assume that $\lambda \notin \overline{\D}$. In turn, $\operatorname{dist}(\lambda, \overline{\D}) : = \delta > 0$. Since $|\phi(z) - \lambda| \ge \delta$ for all $z\in \D$, we have $1/(\phi(z) - \lambda)$ is analytic and bounded by $1/\delta$ on $\D$. Hence, we have $(T_\phi^t - \lambda I)^{-1} = T_{1/(\phi - \lambda)}$, and so $\lambda \notin \sigma(T_\phi^t)$.
 
To see (\ref{main1item2}), note first that $\mathscr{A}_{a,b}\subseteq H^\infty$ and therefore Proposition  \ref{analyticpointspectrum} guarantees that $\sigma_p(T^t_\phi) = \langle \phi k^t_a, k^t_a\rangle/ \|k^t_a\|^2$. However, since $\phi \in \mathscr{A}_{a,b}$, we have $\phi k_a^t\in H^2_t$, and therefore $\langle \phi k^t_a, k^t_a\rangle = \phi(a)k_a^t(a)$. This gives $\sigma_p(T^t_\phi) = \phi(a)$, which is independent of $t$, so we have $\Lambda_\phi^{a,b} = \{\phi(a)\}$.
\end{proof}

Before we turn to real-valued symbols, we provide a brief discussion of arbitrary symbols to establish that $\sigma(T^t_\phi)$ is contained in the closed convex hull of the essential range of $\phi$, a result true for \textit{any} $\phi\in L^\infty$. This result is directly analogous to the classical result by Brown and Halmos for unconstrained Toeplitz operators (see \cite[Corollary~7.19]{banachtech}). An important part of the proof of the  following result is \cite[Lemma 4.5]{pfefferjury}, which observes that $\|\phi\|_\infty = \|T^t_\phi\|$.

In order to show this spectral inclusion, we first note that the closed convex hull of a set $E\subseteq \C$ is the intersection of all open half-planes that contain $E$ (e.g. see \cite[Lemma 7.17]{banachtech}.) We denote the closed convex hull of the essential range of a function $\phi$ by
\[
\mathcal{R}_e(\phi) := \bigcap \left\{S : \text{essran}(\phi)\subseteq S \ \text{and} \ S \ \text{is a half-plane in} \ \C \right\}.
\]
Note that $\mathcal{R}_e(\phi)$ is necessarily a closed subset of $\C$. Thus, if $\phi$ is bounded, the above set is compact. We now have the following spectral inclusion.

\begin{proposition}\label{convexhull}
If $\phi\in L^\infty$, then $\sigma(T^t_\phi)$ is contained in $\mathcal{R}_e(\phi)$. 
\end{proposition}

\begin{proof}
We show that if $\tau\notin \mathcal{R}_e(\phi)$, then $\tau\notin \sigma(T^t_\phi)$. Accordingly, suppose $\tau \in \C \setminus \mathcal{R}_e(\phi)$. Then $\tau \not\in S$ for some open half-plane $S\subseteq \C$ containing $\text{essran}(\phi)$. After translation and rotation, we can assume $S$ is the right half-plane $\{z\in \C : \text{Re}(z)>0\}$. Thus, since $\phi-\tau$ is bounded, we have that $\mathcal{R}_e(\phi-\tau)$ is a compact subset of $S$. Hence, there exists $\varepsilon>0$ such that
\[
\varepsilon\mathcal{R}_e(\phi-\tau) \subseteq \{z\in \C : |z-1|<1\},
\]
and therefore $\|1-\varepsilon (\phi-\tau) \|_\infty < 1$. It is known that $\|1-\varepsilon (\phi-\tau) \|_\infty = \|I - T^t_{\varepsilon(\phi-\tau)} \|$ (see \cite[Lemma~4.5]{pfefferjury}), which gives $\|I - T^t_{\varepsilon(\phi-\tau)} \| <1$. This implies that $T^t_{\varepsilon(\phi-\tau)} = \varepsilon T^t_{\phi-\tau}$ is invertible. Therefore $\tau\notin \sigma(T^t_\phi)$.
\end{proof}

We now turn to results regarding real-valued symbols. Without surprise, this setting poses a greater challenge than the analytic setting.

\section{Toeplitz operators with real-valued symbols} 
 
In this section, we provide a proof of Theorem \ref{mainresult2}. We begin by noting that, for $\phi\in L^\infty$ real-valued, the space $\ker(T^t_\phi)$ is naturally associated with the annihilator of $\mathscr{A}_{a,b} + \mathscr{A}_{a,b}^* :=  \{h_1+\overline{h_2}: h_1,h_2\in \mathscr{A}_{a,b}\}$. In order to establish this fact, we first record a few propositions.

\begin{proposition}\label{getrektd}
Let $\phi \in L^\infty$ be real-valued. If $g \in \ker(T^t_\phi)$, then $\phi |g|^2$ annihilates $\mathscr{A}_{a,b} + \mathscr{A}_{a,b}^*$. 
\end{proposition}

\begin{proof} Suppose $T^t_\phi g=0$ and let $h\in \mathscr{A}_{a,b}$. Then $hg\in H^2_t$ and we have
\[
0 = \langle T^t_\phi g, hg \rangle 
    = \int_\T \phi g \overline{hg}\,d\mu 
    = \int_\T \phi |g|^2 \overline{h}\,d\mu.
\]
As $\phi|g|^2$ is real-valued, we also have $\int_{\T} \phi |g|^2 h \, d\mu = 0$, which establishes the claim.
\end{proof}

We now characterize annihilating measures for $\mathscr{A}_{a,b} + \mathscr{A}_{a,b}^*$, considered as a subspace of $L^\infty$.

\begin{proposition} \label{annihilatorform}
Let $\nu << \mu$ be a measure whose Radon-Nikodym derivative with respect to $\mu$ is in $L^1$.
Then $\frac{d\nu}{d\mu}$ annihilates $\mathscr{A}_{a,b} + \mathscr{A}_{a,b}^*$ if and only if there exists $d_1,d_2 \in \C$ such that 
\[
\frac{d\nu}{d\mu} = d_1(k_a - k_b) + \overline{d_2(k_a - k_b)} \ \ a.e. \ on \ \T.
\]
\end{proposition}

\begin{proof}
We begin with the backward implication. Assume that there exists $d_1, d_2\in \C$ such that $ \textstyle 
\frac{d\nu}{d\mu} = d_1(k_a - k_b) + \overline{d_2(k_a - k_b)}$ and let $h_1+\overline{h_2} \in \mathscr{A}_{a,b} + \mathscr{A}_{a,b}^*$. It follows that
\begin{align*}
\int_\T (h_1+\overline{h_2})\frac{d\nu}{d\mu} \mathop{d\mu} &= \int_\T (h_1+\overline{h_2})(d_1(k_a - k_b) + \overline{d_2(k_a - k_b)}) \mathop{d\mu}\\
&= d_1\int_\T h_1(k_a - k_b) \mathop{d\mu} + \overline{d_2}\int_\T h_1\overline{k_a - k_b}\mathop{d\mu} \\
& \ \ \  + d_1\int_\T \overline{h_2}(k_a - k_b) \mathop{d\mu}+\overline{d_2}\int_\T \overline{h_2}\overline{k_a - k_b} \mathop{d\mu}.
\end{align*}
The first and fourth terms are both zero since $(k_a-k_b)(0)=0$ and the integrands are entirely analytic and anti-analytic, respectively. Further, since $h_1$ and $h_2$ are both taken from $\mathscr{A}_{a,b}$, we have
\begin{align*}
\int_\T (h_1+\overline{h_2})\frac{d\nu}{d\mu} \mathop{d\mu} &=  \overline{d_2}\int_\T h_1\overline{k_a - k_b} \mathop{d\mu} + d_1\int_\T \overline{h_2}(k_a - k_b) \mathop{d\mu} \\
 &=\overline{d_2}\langle h_1,k_a - k_b \rangle + d_1 \overline{\langle h_2, k_a - k_b\rangle } \\
 &=0.
\end{align*}
Thus, $\textstyle \frac{d\nu}{d\mu}$ annihilates $\mathscr{A}_{a,b} + \mathscr{A}_{a,b}^*$.

For the forward implication, assume $\textstyle \frac{d\nu}{d\mu}$ annihilates $\mathscr{A}_{a,b}+ \mathscr{A}_{a,b}^*$ and let $p_n(z) = z^n(z - a)(z - b)$ for $n\geq 0$. Since $p_n$ vanishes at $a$ and $b$, we have $p_n, \overline{p_n} \in \mathscr{A}_{a,b} + \mathscr{A}_{a,b}^*$. Hence,
\[
0 = \int_\T p_n \frac{d\nu}{d\mu} d\mu
= \int_\T z^{n+2}\frac{d\nu}{d\mu} d\mu - \int_\T (a + b)z^{n+1}\frac{d\nu}{d\mu} d\mu + \int_\T abz^n \frac{d\nu}{d\mu} d\mu.
\]
This gives that the negative Fourier coefficients of $\frac{d\nu}{d\mu}$ must satisfy the linear recurrence relation
\[
\widehat{\frac{d\nu}{d\mu}}(n+2) = (a+b)\widehat{\frac{d\nu}{d\mu}}(n+1) - ab\widehat{\frac{d\nu}{d\mu}}(n).
\]
One can readily verify that this recurrence relation is solved by $\widehat{\frac{d\nu}{d\mu}}(n) =c_1 a^n + c_2 b^n$, for some constants $c_1, c_2 \in \C$.  Similarly, considering 
$0 = \int_\T \overline{p_n} \frac{d\nu}{d\mu} d\mu$, we see that the positive Fourier coefficients of $\frac{d\nu}{d\mu}$ must satisfy the recurrence relation 
\[
\widehat{\frac{d\nu}{d\mu}}(n-2) = (\overline{a+b})\widehat{\frac{d\nu}{d\mu}}(n-1) - \overline{ab}\widehat{\frac{d\nu}{d\mu}}(n).
\]
This recurrence relation is solved by $\widehat{\frac{d\nu}{d\mu}}(n) =c_3 \overline{a}^n + c_4 \overline{b}^n$, for some constants $c_3, c_4 \in \C$.
Noting that $\widehat{\frac{d\nu}{d\mu}}(0)=0$, we find that, almost everywhere on $\T$, we have
\begin{align*}
    \frac{d\nu}{d\mu}(z) &= \sum_{n<0} \widehat{\frac{d\nu}{d\mu}}(n) z^n + \sum_{n>0} \widehat{\frac{d\nu}{d\mu}}(n) z^n \\
    &= \sum_{n<0} (c_1 a^n + c_2 b^n) \overline{z}^n + \sum_{n>0} (c_3 \overline{a}^n + c_4 \overline{b}^n )z^n \\
    &= c_1\sum_{n\le0} (a\overline{z})^n - c_1 + c_2\sum_{n\le 0} (b\overline{z})^n - c_2 \\
    &\ \ \ + c_3\sum_{n\ge0} (\overline{a}z)^n - c_3 + c_4\sum_{n\ge0} (\overline{b}z)^n - c_4 \\
    &= c_1 \overline{k_a(z)} - c_1 + c_2 \overline{k_b(z)} - c_2 + c_3k_a(z) - c_3 + c_4k_b(z) - c_4.
\end{align*}
For any $h \in \mathscr{A}_{a,b}$ with $h(0)=0$, we now have
\[ 
0 = \int_\T h \frac{d\nu}{d\mu} d\mu 
= c_3h(a) + c_4h(b).
\]
Since $h(a) = h(b)$, we have $c_4 = -c_3$. Similarly, for $\overline{h}\in \mathscr{A}_{a,b}^*$ with $h(0)=0$, we have
\[ 
0 = \int_\T \overline{h} \frac{d\nu}{d\mu} d\mu 
= c_1\overline{h(a)} + c_2\overline{h(b)}.
\]
Again, since $\overline{h(a)} = \overline{h(b)}$, we have $c_2 = -c_1$. 
Letting $d_1 = c_3$ and $d_2=c_1$, we have 
\[
\frac{d\nu}{d\mu} = d_1(k_a - k_b) + \overline{d_2 ( k_a - k_b)}. \]
\end{proof}

The upshot now is that, by passing to the outer factor $G$ of $g \in \ker(T_\phi^t)$, we have an explicit description of $\phi|G|^2$. First, however, we need the following lemma.

\begin{lemma} \label{density}
If $G\in H^2_t$ is outer, then $G\mathscr{A}_{a,b}$ is dense in $H^2_t$.
\end{lemma}
\begin{proof}
Suppose $f,G\in H^2_t$ with $G$ outer. Since $G$ is outer, there exists a sequence of polynomials $\{p_n\}$ such that $p_nG\rightarrow f$. Let $P_{a,b}:H^2 \rightarrow (H^2_{t=1})^\perp$ be the projection onto the $\text{span}\{k_a-k_b\}$ and define 
\[
q_n :=p_n - P_{a,b}(p_n)=p_n -\ip{p_n}{\frac{k_a-k_b}{\|k_a-k_b\|}}\frac{k_a-k_b}{\|k_a-k_b\|}.
\]
We claim that $q_n G \to f$. Observe that $q_n \in H^2_{t = 1}$, and since $q_n$ is a linear combination of $H^\infty$ functions, we have that $q_n\in \mathscr{A}_{a,b}$. Observe that $q_nG \to f$ when
\[
\ip{p_n}{\frac{k_a-k_b}{\|k_a-k_b\|}}\frac{k_a-k_b}{\|k_a-k_b\|} = (p_n(a) - p_n(b))\frac{k_a - k_b}{\|k_a - k_b\|^2} \rightarrow 0,
\]
which happens if and only if $(p_n(a)- p_n(b))\rightarrow 0$. Since $p_nG\rightarrow f$ in norm, it also converges pointwise. Therefore, 
\[
p_n(a)G(a)\rightarrow f(a)\quad \text{ and }\quad p_n(b)G(b)\rightarrow f(b).
\]
Since $G$ and $f$ are in $H^2_t$ and $G$ is outer (therefore $G(a)\neq 0)$, we have
\begin{align*}
    \lim_{n\rightarrow \infty} (p_n(a)-p_n(b))&=\lim_{n\rightarrow \infty}\frac{1}{G(a)}\left(p_n(a)G(a)-tG(b)p_n(b)\right)\\
    &=\frac{1}{G(a)}\left(f(a)-tf(b)\right) = 0.
\end{align*}
\end{proof}

We now prove the aforementioned identification between $\ker(T^t_\phi)$ and the annihilator measures of $\mathscr{A}_{a,b} + \mathscr{A}_{a,b}^*$.

\begin{proposition} \label{outerimplieseigen}
Let $\phi \in L^\infty$ be real-valued and suppose $G \in H^2_t$ is outer. Then $G \in \ker(T^t_\phi)$ if and only if there exists a constant $c \in \C$ so that 
\[
\phi|G|^2 = c(k_a - k_b) + \overline{c(k_a - k_b)}. 
\]
\end{proposition}
\begin{proof}
We prove the backward direction first. Suppose $G\in H^2_t$ is outer and let $h\in \mathscr{A}_{a,b}$. Observe that
\[
\langle T^t_\phi G, hG \rangle = \int_\T \phi |G|^2 \overline{h} \, d\mu = \int_\T (c(k_a-k_b) + \overline{c(k_a - k_b})) \overline{h} \, d\mu = 0.
\]
In view of Lemma \ref{density}, $T^t_\phi G \equiv 0$.

To see the forward direction, observe that if $T^t_\phi G =0$, then it follows from Proposition \ref{getrektd} that $\phi|G|^2$ annihilates $\mathscr{A}_{a,b}+ \mathscr{A}_{a,b}^*$. Further, since $G\in L^2$ and $\phi \in L^\infty$, we have that $\phi|G|^2\in L^1$. Thus, the measure $\phi |G|^2 \mathop{d\mu}$ is a measure that is absolutely continuous with respect to $\mu$ and whose Radon-Nikodym derivative with respect to $\mu$ is in $L^1$. It now follows from Proposition \ref{annihilatorform} that there exists $d_1,d_2 \in \C$ such that
\[
\phi|G|^2 = d_1(k_a - k_b) + \overline{d_2(k_a - k_b)}.
\]
However, since $\phi|G|^2$ is real-valued, this is only possible if $d_1=d_2$. Putting $c$ as this common value, the result follows.
\end{proof}

\begin{remark}
In the backward direction of Proposition \ref{outerimplieseigen}, we require $G$ to be outer in order to use Lemma \ref{density}. However, the forward direction does not require $G$ to be outer. 
\end{remark}

We will now have an interlude to discuss the behavior of $c(k_a-k_b) + \overline{c(k_a-k_b)}$ on $\T$. 

\begin{lemma}\label{realequality}
For any choice of $a,b \in \D$ and $c \in \C\setminus\{0\}$, the function 
\[
\text{Re}\left
(c(k_a(e^{it}) - k_b(e^{it}))\right)
\]
is positive (negative) on precisely one proper sub-arc of $\T$. 
\end{lemma}

\begin{proof}
Begin by letting $u(z) :=  k_a(z) - k_b(z)$ and noting
\[
u(z) = \frac{\overline{(a - b)}z}{(1 - \overline{a}z)(1 - \overline{b}z)}.
\]
In turn, we see that $cu$ fixes the origin for any value of $c \in \C$ and is analytic on a disk containing $\overline{\D}$. By the open mapping theorem, we have $0$ is in the interior of $cu(\D)$. Further, and again by the open mapping theorem, we have that the boundary of $cu(\D)$ is contained in $cu(\T)$. But since $0$ is in the interior of $cu(\D)$, it must be that the boundary of $cu(\D)$ has a component both in the left-half and right-half plane, and so must $cu(\T)$. 
Also notice that 
\begin{align*}
\text{Re}(cu(z)) &= \frac12\left[\frac{c\overline{(a - b)}z}{(1 - \overline{a}z)(1 - \overline{b}z)} + \frac{\overline{c}(a - b)\overline{z}}{(1 - a\overline{z})(1 - b\overline{z})}\right]\\
&= \frac{c\overline{(a - b)}z(1 - a\overline{z})(1 - b\overline{z}) + \overline{c}(a - b)\overline{z}(1 - \overline{a}z)(1 - \overline{b}z)}{2\left|(1 - \overline{a}z)\right|^2 \left|(1 - \overline{b}z)\right|^2} .
\end{align*}
On the circle, $\text{Re}(cu(e^{it})) = 0$ is a homogeneous trigonometric polynomial equation of degree one, so $\text{Re}(cu(e^{it}))$ has at most two zeros. Equivalently, $cu(\T)$ is purely imaginary at most twice. But since $cu(\T)$ has a component in both the left-half and right-half plane, it follows that $\text{Re}(cu(e^{it}))$ has precisely two zeros, and changes sign precisely twice. The result follows.
\end{proof}

We have one more important observation to make about the behavior of the zeros of $\text{Re}(c(k_a - k_b))$ on the circle.

\begin{proposition}\label{no_contain}
For any constants $c, d \in \C$, with $c$ not a non-negative multiple of $d$, we have on $\T$ that
\[
\{   d(k_a - k_b) + \overline{d(k_a - k_b)} > 0 \} \not\subset \{   c(k_a - k_b) + \overline{c(k_a - k_b)} > 0 \}.
\] 
\end{proposition}

\begin{proof}
Begin by letting $u(z) :=  k_a(z) - k_b(z)$ and observing that $\{ du + \overline{du} > 0 \} = \{ cu + \overline{cu} > 0 \}$ if and only if $c$ and $d$ are non-negative real multiples. Suppose for contradiction that
\[
\{ du + \overline{du} > 0 \} \subsetneq \{ cu + \overline{cu} > 0 \}.
\] 
Then, letting $t_c = \operatorname{Arg}c$ and $t_d = \operatorname{Arg}d$, we have
\[
\{ e^{it_d}u + \overline{e^{it_d}u} > 0 \} \subsetneq \{ e^{it_c}u + \overline{e^{it_c}u} > 0 \}.
\] 

Let $e^{i\theta_{c_1}}$ and $e^{i\theta_{c_2}}$ be the solutions to $cu + \overline{cu} = 0$ on $\T$. Similarly, let $e^{i\theta_{d_1}}$ and $e^{i\theta_{d_2}}$ be the solutions to $du + \overline{du} = 0$ on $\T$. By the containment hypothesis and Lemma \ref{realequality}, we have $e^{i\theta_{c_1}}, e^{i\theta_{c_2}} \in \{ du + \overline{du} \le 0 \}$ (we include equality with zero as two of the roots $e^{i\theta_{c_j}}, e^{i\theta_{d_j}}$, $j=1,2$, may be equal). 
Without loss of generality, this implies that $\text{Re}(e^{it_d}u(e^{i\theta_{c_1}})) \le 0$ and $\text{Re}(e^{it_d}u(e^{i\theta_{c_2}})) < 0$. 

By definition, we also have that $\text{Re}(e^{it_c}u(e^{i\theta_{c_1}})) = \text{Re}(e^{it_c}u(e^{i\theta_{c_2}})) = 0$. Further, since $0 \in e^{it_c}u(\overline{\D})$, we have that $\text{Im}(e^{it_c}u(e^{i\theta_{c_1}}))$ and  $\text{Im}(e^{it_c}u(e^{i\theta_{c_2}})) $ have different signs. It now follows that for any $\theta \in \R$ not a multiple of $\pi$, the signs of $\text{Re}(e^{i\theta}e^{it_c}u(e^{i\theta_{c_1}}))$ and  $\text{Re}(e^{i\theta}e^{it_c}u(e^{i\theta_{c_2}}))$ are different. In particular, taking $\theta = t_d - t_c$ (by hypothesis not a multiple of $\pi$), we see that $\text{Re}(e^{i(t_d - t_c)}e^{it_c}u(e^{i\theta_{c_1}}))$ and  $\text{Re}(e^{i(t_d - t_c)}e^{it_c}u(e^{i\theta_{c_2}}))$ have different signs. But 
\[
\text{Re}(e^{i(t_d - t_c)}e^{it_c}u(e^{i\theta_{c_1}})) = \text{Re}(e^{it_d}u(e^{i\theta_{c_1}}))
\]
and
\[
\text{Re}(e^{i(t_d - t_c)}e^{it_c}u(e^{i\theta_{c_2}})) = \text{Re}(e^{it_d}u(e^{i\theta_{c_2}})).
\]
This is a contradiction.
\end{proof}

For $\phi \in L^\infty$, recall that the set of eigenvalues of $T_\phi^t$ relative to $\mathscr{A}_{a,b}$ is 
\[
\Lambda_\phi^{a,b} := \bigcup_{t\in \widehat{\C}} \sigma_p(T^{t}_\phi).
\]
We now identify  $\Lambda_\phi^{a,b}$ with annihilators of $\mathscr{A}_{a,b} + \mathscr{A}_{a,b}^* :=  \{h_1+\overline{h_2}: h_1,h_2\in \mathscr{A}_{a,b}\}$.

\begin{proposition}\label{ann_rep}
Let $\phi \in L^\infty$ be real-valued. There exists a $c\in \C$ such that for every $\lambda\in \Lambda_\phi^{a,b}$, there exists an outer function $G_\lambda$ such that
\[
(\phi - \lambda) |G_\lambda|^2 = c(k_a  - k_b) + \overline{c(k_a  - k_b)}.
\]
Moreover, the constant $c$ is unique up to a non-negative real multiple.
\end{proposition}

\begin{proof}
By definition, if $\lambda \in \Lambda_\phi^{a,b}$, then there exists $t \in \widehat{\C}$ and $g_\lambda \in H^2_t$ such that $T_\phi^t g_\lambda = \lambda g_\lambda$, or, equivalently, $g_\lambda \in \ker(T_{\phi-\lambda}^t)$. Further, by Proposition \ref{praisejesus}, we know the outer part of $g_\lambda$, say $G_\lambda$,  belongs to $\ker(T_{\phi-\lambda}^s)$, where $s = G_\lambda(a)/G_\lambda(b)$. Now, use Proposition \ref{outerimplieseigen} applied to $\phi - \lambda$ to get the desired result. 

To see that $c$ is unique up to a non-negative constant, let  $\lambda_1, \lambda_2 \in \Lambda_\phi^{a,b}$ be distinct with
\[
(\phi - \lambda_j) |G_{\lambda_j}|^2 = c_j(k_a  - k_b) + \overline{c_j(k_a  - k_b)}, \ \ \ j=1,2.
\] 
Without loss of generality, suppose $\lambda_1 > \lambda_2$. Then $\{ \phi > \lambda_1 \} \subseteq \{ \phi > \lambda_2 \}$, or, equivalently, $\{c_1(k_a  - k_b) + \overline{c_1(k_a  - k_b)} > 0\} \subseteq \{ c_2(k_a  - k_b) + \overline{c_2(k_a  - k_b}) >0 \}$. It now follows from Proposition \ref{no_contain} that $c_1$ and $c_2$ must be non-negative real multiples of each other. 
\end{proof}

We now note that the collection of outer functions in the kernel of $T^t_\phi$, for fixed real-valued $\phi$, is quite small. In fact, all such outer functions are essentially unique.

\begin{proposition} \label{outerinkerneltheorem}
Let $\phi \in L^\infty$ be real-valued and let $\mathcal{N}$ denote the collection of outer functions in $H^2$. Then there is at most one $t \in \widehat{\C}$ so that $\ker(T_\phi^t) \cap \mathcal{N}$ is non-empty. Moreover, when $\ker(T_\phi^t) \cap  \mathcal{N}$ is non-empty, it is equal to the span of a single outer function.
\end{proposition}
\begin{proof}
Suppose $G_1\in \ker(T^t_\phi)$ and $G_2\in \ker(T^s_\phi)$ are outer, with $s \neq t$. From Proposition \ref{outerimplieseigen}, there exist constants $c_1,c_2\in \C$ such that
\[
\phi|G_j|^2 = c_j(k_a-k_b) + \overline{c_j(k_a - k_b)}, \ \ \ j = 1,2.
\]
Since $|G_1|^2$ and $|G_2|^2$ are non-negative, $c_1(k_a-k_b) + \overline{c_1(k_a - k_b)}$ and $c_2(k_a-k_b) + \overline{c_2(k_a-k_b)}$ must be positive (negative) on the same subsets of $\T$. Proposition \ref{no_contain} indicates that this can happen if and only if $c_1$ is a non-negative real multiple of $c_2$. In turn, there exists a positive constant $r \in \R$ so that $\phi|G_1|^2 =  \phi r^2|G_2|^2$. Since outer functions are determined by their modulus on the unit circle, it follows that $G_1$ and $G_2$ are constant multiples of one another and therefore $G_2 \in \operatorname{span} \{G_1\}$. Hence $G_2 \in \ker(T_\phi^t) \cap \mathcal{N}$.

\end{proof}

Marching toward our proof of Theorem \ref{mainresult2}, we pause to record a lemma that will be used for a key proposition.

\begin{lemma}[{\cite[p. 146]{theoryofhb}}] \label{outercreation}
Given $w\in L^p(\T)$ such that $\log(|w|)\in L^1(\T)$, the following function
\[
W(z) = \exp\left( \frac{1}{2\pi}\int_0^{2\pi} \frac{e^{i\theta}+z}{e^{i\theta}-z}\, \log(|w(e^{i\theta})|) \mathop{d\theta}\right),
\]
defined for $z\in \D$, is an outer function in $H^p$. Moreover, $|W|=|w|$ on $\T$.
\end{lemma}

We now introduce some important notation. For $c\in \C$, let
\[
S_c^- := \{ z\in \T : c(k_a-k_b)+\overline{c(k_a-k_b)} <0\}
\]
and
\[
S_c^+ := \{ z\in \T : c(k_a-k_b)+\overline{c(k_a-k_b)} >0 \}.
\]
It turns out that these sets are crucial in understanding the relative eigenvalues of $T_\phi^t$.

\begin{proposition} \label{hitter}
Let $\phi \in L^\infty$ be real-valued. If there exist $c \in \C$ and $\beta\in \R$ such that 
            \[
            \esssup\{ (\phi - \beta)\rvert_{S_c^-} \} = m < 0 < M = \essinf\{ (\phi-\beta)\rvert_{S_c^+} \},
            \]
then
\begin{enumerate}[(i)]
    \item \label{hitter1} $(m + \beta, M + \beta) \subseteq \Lambda_\phi^{a,b}$.
    \item \label{hitter2} For every $\beta+\lambda \in (m + \beta, M + \beta)$, there exists an essentially unique outer function $G_{\beta,\lambda}$ such that
        \[
        (\phi - (\beta+\lambda)) |G_{\beta,\lambda}|^2 = c(k_a-k_b) + \overline{c(k_a-k_b)}.
        \]
\item The endpoints $M+\beta$ and $m+\beta$ are elements of $\Lambda_\phi^{a,b}$  if and only if 
\[
\frac{c(k_a-k_b) + \overline{c(k_a-k_b)}}{\phi - (M+\beta)} \textrm{\ \ \ and \ \ \ } \frac{c(k_a-k_b) + \overline{c(k_a-k_b)}}{\phi - (m+\beta)} 
\]
are integrable.
\end{enumerate}
\end{proposition}

\begin{proof} 
Let $\lambda \in (m,M)$ and consider the function
\[
\psi:= \frac{c(k_a-k_b) + \overline{c(k_a-k_b)}}{\phi - (\beta +\lambda)}.
\]
By definition of $M$, on $S_c^+$ we have, $\lambda < M \leq \phi - \beta$ and hence $0<\phi - (\beta+\lambda)$. Additionally, $\phi - (\beta+\lambda)$ is uniformly bounded away from $0$ on $S_c^+$. Similarly, by definition of $m$, the function $\phi-(\beta+\lambda)$ is negative and uniformly bounded away from zero on $S_c^-$. In turn, $|\phi - (\beta+\lambda)|$ is uniformly bounded away from zero. Further, by definition of $S_c^+$ and $S_c^-$, it is clear that $\psi$ is non-negative.

Now consider
\[
\log(\psi) = \log(|c(k_a-k_b) + \overline{c(k_a-k_b)}|) - \log(|\phi - (\beta+\lambda)|).
\]
Note first that $\int_\T \log(|\phi - (\beta+\lambda) |) \mathop{dm} < \infty$, so that 
\[
-\int_\T \log(|\phi - (\beta+\lambda) |) \mathop{dm} > -\infty.
\] 
It follows from Lemma \ref{realequality} that 
\[
\int_\T \log(|c(k_a-k_b) + \overline{c(k_a-k_b)}|) \mathop{dm} > -\infty.
\] 
Altogether, we find that
\[
\int_\T \log(\psi) \mathop{dm}> -\infty.
\]
Thus, since $\T$ is compact, it now follows that $\log(\psi)\in L^1$. Now consider the function $\psi^{1/2}$. Since $\psi \in L^1$ and is non-negative, we have that $\psi^{1/2} \in L^2$. Further, $\log(\psi^{1/2}) = \textstyle \frac{1}{2}\log(\psi) \in L^1$.

Now, since $\psi^{1/2}\in L^2$ and $\log(\psi^{1/2})\in L^1$, it follows from Lemma \ref{outercreation} that the function
\[
G_{\beta,\lambda}(z):= \exp\left( \frac{1}{2\pi}\int_0^{2\pi} \frac{e^{i\theta}+z}{e^{i\theta}-z} \log(|\psi^{1/2}(e^{i\theta})|) \mathop{d\theta}\right) 
\]
is an outer function in $H^2$ with $|G_{\beta,\lambda}|=\psi^{1/2}$ on $\T$. Thus, we have an outer function $G_{\beta,\lambda} \in H^2$ with
\[
|G_{\beta,\lambda}|^2 = \psi =  \frac{c(k_a-k_b) + \overline{c(k_a-k_b)}}{\phi-(\beta+\lambda)},
\]
where the last equality follows from $\psi$ being non-negative. 

It now follows from Proposition \ref{outerimplieseigen} that $T^t_{\phi-(\beta+\lambda)} G_{\beta,\lambda} = 0$ for $t = \textstyle \frac{G_{\beta,\lambda}(a)}{G_{\beta,\lambda}(b)}$ (where $G_{\beta,\lambda}(b)\neq 0$ since $G_{\beta,\lambda}$ is outer). In light of Proposition \ref{outerinkerneltheorem}, we have that $G_{\beta,\lambda}$ is unique up to a constant. Rewriting $T^t_{\phi-(\beta+\lambda)} G_{\beta,\lambda} = 0$, we find that $T^t_\phi G_{\beta,\lambda} = (\lambda+\beta) G_{\beta,\lambda}$ and therefore $\lambda+\beta$ is an eigenvalue for $T^t_\phi$. This means $\lambda+\beta \in \Lambda_\phi^{a,b}$ and therefore $(m + \beta, M + \beta) \subseteq \Lambda_\phi^{a,b}$.

If we require the integrability of  
\[
\frac{c(k_a-k_b) + \overline{c(k_a-k_b)}}{\phi - (M+\beta)} \textrm{\ \ \ and \ \ \ } \frac{c(k_a-k_b) + \overline{c(k_a-k_b)}}{\phi - (m+\beta)},
\]
the above argument holds for $\lambda = M$ and $\lambda = m$. 
\end{proof}

Finally, we have the machinery necessary to prove Theorem \ref{mainresult2}.

\begin{theorem} \label{mainresult2} 
If $\phi\in L^\infty$ is real-valued, then $\Lambda_\phi^{a,b}$ is either empty, a point, or an interval.
In particular, $\Lambda_\phi^{a,b}$ is connected. 
\end{theorem}
\begin{proof}
If $\Lambda_\phi^{a,b}$ is empty or a point, then we are done. Suppose, then, $\Lambda_\phi^{a,b} \neq \emptyset$ and that there exist distinct $\lambda_1,\lambda_2\in \Lambda_\phi^{a,b}$. By Proposition \ref{ann_rep}, there exist outer functions $G_1, G_2 \in H^2$ and $c_1,c_2 \in \C$ such that
\begin{equation*} 
    (\phi - \lambda_j)|G_j|^2 = c_j(k_a-k_b) + \overline{c_j(k_a-k_b)}, \ \ \ j = 1,2,
\end{equation*}
where $c_1$ and $c_2$ differ by a non-negative real constant. By absorbing this constant into $|G_j|^2$ and relabeling appropriately, it follows that there exists a single $c\in \C$ such that
\begin{equation*} 
    (\phi - \lambda_j)|G_j|^2 = c(k_a-k_b) + \overline{c(k_a-k_b)}, \ \ \ j = 1,2.
\end{equation*}
 Without loss of generality, assume $\lambda_1 < \lambda_2$ and let $\beta \in (\lambda_1,\lambda_2)$. Then we have
\[
\phi-\lambda_2 < \phi-\beta < \phi-\lambda_1.
\]
Therefore
\[
(\phi-\beta)\rvert_{S_c^-} < (\phi-\lambda_1)\rvert_{S_c^-} < 0
\]
and
\[
0<(\phi-\lambda_2)\rvert_{S_c^+}<(\phi-\beta)\rvert_{S_c^+}.
\]
Finally, it follows from Lemma \ref{realequality} that the sets $S_c^{\pm}$ are of positive measure. As a result, we can define the constants
\[
m:=\esssup\{ (\phi - \beta)\rvert_{S_c^-} \}
\]
\[
M:=\essinf\{ (\phi-\beta)\rvert_{S_c^+} \}
\]
and conclude that $m<0<M$. It follows from Proposition \ref{hitter} that $(m + \beta, M + \beta)\subseteq \Lambda_\phi^{a,b}$ (with inclusion of endpoints possible).

We claim that this interval is all of $\Lambda_\phi^{a,b}$ and show that if $\lambda+\beta > M$ (resp. $\lambda+\beta<m$), then $\lambda+\beta \notin \Lambda_\phi^{a,b}$. To this end, suppose $\lambda+\beta > M+\beta$, and, for the sake of contradiction, that $\lambda+\beta \in \Lambda_\phi^{a,b}$. By Proposition \ref{outerimplieseigen},
there exists a constant $d\in \C$ and $G_{\beta+\lambda} \in H^2$ outer such that
\[
(\phi-(\beta+\lambda))|G_{\beta+\lambda}|^2 = d(k_a-k_b) + \overline{d(k_a-k_b)}.
\]
Similarly, since $0\in (m,M)$, we have $\beta \in (m+\beta,M+\beta)\subseteq \Lambda_\phi^{a,b}$ and therefore, by part (\ref{hitter2}) of Proposition \ref{hitter},
\[
(\phi-\beta)|G_\beta|^2 = c(k_a-k_b) +\overline{c (k_a-k_b)}
\]
for some outer function $G_\beta \in H^2$. 
Now, since $\lambda+\beta > M+\beta$, we have
\begin{equation*} \label{bigcon}
    \{   \phi-(\beta+\lambda)>0\} \subsetneq \{ \phi-\beta>0\}
\end{equation*}
and therefore
\begin{align*}
   \{ d(k_a-k_b) + \overline{d(k_a-k_b)}>0\} &=  \{  \phi - (\beta+\lambda)>0\}\\
    &\subsetneq\{  \phi-\beta>0\} \\
    &= \{  c (k_a-k_b)+\overline{c (k_a-k_b)}>0\}.
\end{align*}
However, by Proposition \ref{no_contain}, this inclusion cannot occur. Hence we arrive at a contradiction and $\lambda+\beta \notin  \Lambda_\phi^{a,b}$. An analogous arguments holds for $\lambda+\beta < m+\beta$.
\end{proof}

\section{Further Remarks and Questions}
For real-valued symbols, the proof of Theorem \ref{mainresult2} shows, when expecting a relative point spectrum consisting of an interval, the symbol of the Toeplitz operator is quite confined; namely, the hypotheses of Proposition \ref{hitter} must be satisfied. However, there is no mention of explicit hypotheses required on the symbol to guarantee the relative point spectrum to be empty or a point -- the former of which we argue is most often the case. In fact, it can be shown, using arguments very similar to the ones used in the above proofs, that $\Lambda^{a,b}_\phi$ is a point if and only if there exist constants $c \in \C$ and $\beta\in \R$ such that 
\[
\esssup\{ (\phi - \beta)\rvert_{S_c^-} \} =  0  = \essinf\{ (\phi-\beta)\rvert_{S_c^+} \}.
\]
Again, the behavior of $\phi$ is quite restricted. All told, in order for a constrained Toeplitz operator to have non-empty relative spectrum in the real-valued setting, it must have a symbol (up to a translation) with set of positivity (negativity) coinciding with that of $\textrm{Re}(c(k_a - k_b))$ for some $c\in \C$. 

Reflecting back to Section \ref{intro}, Widom's connected-spectrum result held for arbitrary symbols. A result of this generality is still unknown for the two-point and Neil cases:
\begin{question}[Open]
For general $\phi \in L^\infty$, are the spectrum, essential spectrum, and relative point spectrum of $T_\phi^t$ connected? In the Neil algebra setting? 
\end{question}

The Neil and two-point constraints are, heuristically speaking, the building blocks for general finite codimensional subalgebras of $H^\infty$. This notion, due to Gamelin, was formalized in \cite[Theorem 9.8]{Gamelin1}, and reinterpreted in \cite[Theorem 2.1]{pfefferjury}. A natural generalization of the work here, and in \cite{broschinski2014eigenvalues}, would be to consider spectra of Toeplitz operators associated to these algebras:

\begin{question}[Open]
For a general finite codimensional subalgebra of $H^\infty$, do the associated Toeplitz operators have connected relative spectrum?
\end{question}
Here, by `associated' we mean Toeplitz operators acting on subspaces of $H^2$ which carry a representation for the sublagebra.

\subsection*{Statements and Declarations}
The authors declare that they have no conflict of interest.
\subsection*{Acknowledgments}
The authors thank John McCarthy and Scott McCullough for helpful discussion, as well as the reviewer for their thoughtful suggestions.

\bibliography{refs}

\begin{thebibliography}{10}

\bibitem{AbrahamseToeplitz}
M.~B. Abrahamse.
\newblock Toeplitz operators in multiply connected regions.
\newblock {\em Amer. J. Math.}, 96:261--297, 1974.

\bibitem{Mult-ConnDom}
M.~B. Abrahamse and R.~G. Douglas.
\newblock A class of subnormal operators related to multiply-connected domains.
\newblock {\em Advances in Math.}, 19(1):106--148, 1976.

\bibitem{akbari1}
G.~Akbari~Estahbanati.
\newblock On the spectral character of {T}oeplitz operators on planar regions.
\newblock {\em Proc. Amer. Math. Soc.}, 124(9):2737--2744, 1996.

\bibitem{Anderson-ROchberg}
J.~M. Anderson and R.~H. Rochberg.
\newblock Toeplitz operators associated with subalgebras of the disk algebra.
\newblock {\em Indiana Univ. Math. J.}, 30(6):813--820, 1981.

\bibitem{aryana1}
C.~P. Aryana.
\newblock Self-adjoint {T}oeplitz operators associated with representing
  measures on multiply connected planar regions and their eigenvalues.
\newblock {\em Complex Anal. Oper. Theory}, 7(5):1513--1524, 2013.

\bibitem{clancey2}
C.~P. Aryana and K.~F. Clancey.
\newblock On the existence of eigenvalues of {T}oeplitz operators on planar
  regions.
\newblock {\em Proc. Amer. Math. Soc.}, 132(10):3007--3018, 2004.

\bibitem{MR2743417}
S.~Axler.
\newblock Toeplitz operators.
\newblock In {\em A glimpse at {H}ilbert space operators}, volume 207 of {\em
  Oper. Theory Adv. Appl.}, pages 125--133. Birkh\"{a}user Verlag, Basel, 2010.

\bibitem{balasubramanian2019szegHo}
S.~Balasubramanian, S.~McCullough, and U.~Wijesooriya.
\newblock Szeg{\H{o}} and {W}idom theorems for the {N}eil algebra.
\newblock In {\em Interpolation and realization theory with applications to
  control theory}, volume 272 of {\em Oper. Theory Adv. Appl.}, pages 61--70.
  Birkh\"{a}user/Springer, Cham, 2019.

\bibitem{Ballconstrained}
J.~A. Ball, V.~Bolotnikov, and S.~ter Horst.
\newblock A constrained {N}evanlinna-{P}ick interpolation problem for
  matrix-valued functions.
\newblock {\em Indiana Univ. Math. J.}, 59(1):15--51, 2010.

\bibitem{Banjade}
D.~P. Banjade.
\newblock Estimates for the corona theorem on
  {$H^{\infty}_{\mathbb{I}}(\mathbb{D})$}.
\newblock {\em Oper. Matrices}, 11(3):725--734, 2017.

\bibitem{broschinski2014eigenvalues}
A.~Broschinski.
\newblock Eigenvalues of {T}oeplitz operators on the annulus and {N}eil
  algebra.
\newblock {\em Complex Analysis and Operator Theory}, 8(5):1037--1059, 2014.

\bibitem{brownhalmos}
A.~Brown and P.~R. Halmos.
\newblock Algebraic properties of {T}oeplitz operators.
\newblock {\em J. Reine Angew. Math.}, 213:89--102, 1963/64.

\bibitem{contsymbol2}
A.~Calder\'{o}n, F.~Spitzer, and H.~Widom.
\newblock Inversion of {T}oeplitz matrices.
\newblock {\em Illinois J. Math.}, 3:490--498, 1959.

\bibitem{clancey1}
K.~F. Clancey.
\newblock On the spectral character of {T}oeplitz operators on multiply
  connected domains.
\newblock {\em Trans. Amer. Math. Soc.}, 323(2):897--910, 1991.

\bibitem{cowen2021convexity}
C.~C. Cowen and C.~Felder.
\newblock Convexity of the {B}erezin range.
\newblock {\em arXiv:2109.12095}, 2021.

\bibitem{DPRS}
K.~R. Davidson, V.~I. Paulsen, M.~Raghupathi, and D.~Singh.
\newblock A constrained {N}evanlinna-{P}ick interpolation problem.
\newblock {\em Indiana Univ. Math. J.}, 58(2):709--732, 2009.

\bibitem{dev1}
A.~Devinatz.
\newblock Toeplitz operators on {$H^{2}$} spaces.
\newblock {\em Trans. Amer. Math. Soc.}, 112:304--317, 1964.

\bibitem{banachtech}
R.~G. Douglas.
\newblock {\em Banach algebra techniques in operator theory}, volume 179 of
  {\em Graduate Texts in Mathematics}.
\newblock Springer-Verlag, New York, second edition, 1998.

\bibitem{dritschelconstrained}
M.~A. Dritschel and B.~Undrakh.
\newblock Rational dilation problems associated with constrained algebras.
\newblock {\em J. Math. Anal. Appl.}, 467(1):95--131, 2018.

\bibitem{theoryofhb}
E.~Fricain and J.~Mashreghi.
\newblock {\em The theory of {$\mathcal{H}(b)$} spaces. {V}ol. 1}, volume~20 of
  {\em New Mathematical Monographs}.
\newblock Cambridge University Press, Cambridge, 2016.

\bibitem{Gamelin1}
T.~W. Gamelin.
\newblock Embedding {R}iemann surfaces in maximal ideal spaces.
\newblock {\em J. Functional Analysis}, 2:123--146, 1968.

\bibitem{pwcontsymbol1}
I.~C. Gohberg and N.~J. Krupnik.
\newblock The algebra generated by the {T}oeplitz matrices.
\newblock {\em Funkcional. Anal. i Prilo\v{z}en.}, 3(2):46--56, 1969.

\bibitem{halmos}
P.~R. Halmos.
\newblock A glimpse into {H}ilbert space.
\newblock In {\em Lectures on {M}odern {M}athematics, {V}ol. {I}}, pages 1--22.
  Wiley, New York, 1963.

\bibitem{hw2}
P.~Hartman and A.~Wintner.
\newblock On the spectra of {T}oeplitz's matrices.
\newblock {\em Amer. J. Math.}, 72:359--366, 1950.

\bibitem{hw1}
P.~Hartman and A.~Wintner.
\newblock The spectra of {T}oeplitz's matrices.
\newblock {\em Amer. J. Math.}, 76:867--882, 1954.

\bibitem{contsymbol1}
M.~G. Kre\u{\i}n.
\newblock Integral equations on the half-line with a kernel depending on the
  difference of the arguments.
\newblock {\em Uspehi Mat. Nauk}, 13(5 (83)):3--120, 1958.

\bibitem{pfefferjury}
D.~T. Pfeffer and M.~T. Jury.
\newblock Szeg{\H{o}} and {W}idom theorems for finite codimensional subalgebras
  of a class of uniform algebras.
\newblock {\em Complex Anal. Oper. Theory}, 15(5):Paper No. 83, 37, 2021.

\bibitem{Raghupathi2point}
M.~Raghupathi.
\newblock Nevanlinna-{P}ick interpolation for {$\mathbb{C}+BH^\infty$}.
\newblock {\em Integral Equations Operator Theory}, 63(1):103--125, 2009.

\bibitem{Trent1}
J.~Ryle and T.~T. Trent.
\newblock A corona theorem for certain subalgebras of {$H^\infty(\mathbb D)$}.
\newblock {\em Houston J. Math.}, 37(4):1211--1226, 2011.

\bibitem{Trent2}
J.~Ryle and T.~T. Trent.
\newblock A corona theorem for certain subalgebras of {$H^\infty(\mathbb D)$}
  {II}.
\newblock {\em Houston J. Math.}, 38(4):1277--1295, 2012.

\bibitem{contsymbol3}
H.~Widom.
\newblock Inversion of {T}oeplitz matrices. {II}.
\newblock {\em Illinois J. Math.}, 4:88--99, 1960.

\bibitem{widomspectrum}
H.~Widom.
\newblock On the spectrum of a {T}oeplitz operator.
\newblock {\em Pacific J. Math.}, 14:365--375, 1964.

\bibitem{widomspectrum2}
H.~Widom.
\newblock Toeplitz operators on {$H_{p}$}.
\newblock {\em Pacific J. Math.}, 19:573--582, 1966.

\bibitem{analyticsymbol1}
A.~Wintner.
\newblock Zur {T}heorie der beschr\"{a}nkten {B}ilinearformen.
\newblock {\em Math. Z.}, 30(1):228--281, 1929.

\end{thebibliography}
\bibliographystyle{abbrv}
\end{document}